\documentclass[11pt]
{amsart}
\usepackage{amssymb}
\usepackage{dsfont}
\usepackage{color}
\newtheorem{theorem}{Theorem}[section]
\newtheorem{problem}[theorem]{Problem}
\newtheorem{proposition}[theorem]{Proposition}
\newtheorem{corollary}[theorem]{Corollary}
\newtheorem{lemma}[theorem]{Lemma}

\newtheorem{remark}[theorem]{Remark}
\newtheorem{definition}[theorem]{Definition}

\newcommand{\T}{\mathbb{T}}
\newcommand{\Z}{\mathbb{Z}}
\newcommand{\N}{\mathbb{N}}

\def\cont{\mathfrak c}

\catcode`\@=12
\def\T{{\mathbb T}}
\def\eps{{\varepsilon}}

\def\I{{\mathbb I}}
\def\Z{{\mathbb Z}}
\def\N{{\mathbb N}}
\def\R{{\mathbb R}}

\def\cont{\mathfrak c}

\makeindex

\begin{document}

\title[$s$-characterized subgroups (II): continued fractions]
% note on $s$-characterized subgroups for arithmetic sequences]
%{A note on $s$-characterized subgroups for arithmetic sequences}
{Statistically characterized subgroups of the circle  (II): continued fractions}
\subjclass[2010]{Primary:  22B05, 11J70, Secondary: 40A05} \keywords{ Circle group, continued fraction, diophantine approximations, Natural density, statistical convergence, s-characterized subgroup}

\author{Pratulananda Das, Kumardipta Bose}

\address{Department of Mathematics, Jadavpur University, Kolkata-700032, India}
\email {pratulananda@yahoo.co.in, gabuaktafaltu6ele@hotmail.com}

\begin{abstract}

In this note, we  continue the investigation of the new version of characterized subgroups of the circle group $\T$, namely, "statistically characterized subgroups" (shortly, "s-characterized subgroups") recently introduced in \cite{DDB}. We primarily investigate these subgroups for sequences arising out of continued fraction representation of irrational numbers $\alpha$ in line of \cite{L} and \cite{KL} (followed by \cite{BDMW1}) comparing their main results for this new notion and show that these subgroups are strictly larger in size (so nontrivial) than the corresponding characterized subgroups, having cardinality $\mathfrak{c}$ and containing the subgroup $\langle \alpha \rangle$ and in the process answer the Open Question 6.4 posed in \cite{DDB}.
\end{abstract}
\maketitle

\section{Introduction and preliminaries}
Throughout $\mathbb{R}, \mathbb{Q}$ and $\mathbb{N}$ will stand for the set of all real numbers, the set of all rational numbers and the set of all natural numbers respectively. Also, for $x\in\R$ we denote by $\lfloor x\rfloor$ the greatest integer less than $x$, and by $\{x\}$ we denote its fractional part.

The motivation to study the so called "characterized subgroups" can be traced back to the distribution of
sequences of multiples of a given real number mod 1.  Recall that a sequence of real numbers $(x_n)$ is said to be {\em uniformly distributed mod 1}, if for every $[a, b] \subseteq [0, 1)$ one has
$$ \displaystyle{\lim_{n \rightarrow \infty}} \frac{|\{j: 0 \leq j < n, ~\{x_j\} \in [a, b]\}|}{n} = b - a.$$
In his celebrated results proved in 1916, H. Weyl \cite{W} had investigated the set
$$\textit{W}_{\textbf{u}} = \{x \in [0, 1]: (u_nx)~\mbox{is uniformly distributed mod}~1\}$$
where $\textbf{u} = (u_n) \in \mathbb{Z}^{\mathbb{N}}$. Note that for a number $\alpha \in [0, 1] \setminus \mathbb{Q}$, $\alpha \notin \textit{W}_{\textbf{u}}$ for appropriate choice of $\textbf{u}$. The interest to the case when these sequences are small, actually
null sequences, stems from Harmonic Analysis where $A$-sets (short for Arbault sets) were introduced in \cite{A1}
and stems from trigonometric series (see also \cite{Bu,BuKR,El}).

 The interest for the same type of sets also came from another direction initiated by Armacost \cite{A} which culminated in the study of characterized subgroups, in particular,  for arithmetic sequences which have a rich history (one can see \cite{DGDis} and the very recent paper \cite{DDB}).

We now recall the definition of a characterized subgroup  of $\T$.

\begin{definition}
Let $(a_n)$ be a sequence of integers, the subgroup
$$
t_{(a_n)}(\T) := \{x\in \T: a_nx \to 0\mbox{ in } \T\}.
$$
of $\T$ is called {\em a characterized} $($by $(a_n))$ {\em subgroup} of $\T$.
\end{definition}

The other line of investigation of characterized subgroups have been more number theoretic, the one, using continued fraction.
Let $\alpha$ be an irrational number with the regular continued fraction approximation $\alpha=[a_0; a_1,a_2,...]$. For any $n\in\N$, let $\frac{p_n}{q_n} = [a_0; a_1,a_2,...,a_n]$ be the sequence of convergents and we write $\theta_n=q_n\alpha -p_n$. There have been a lot of interest in the characterized subgroup generated by the sequence of denominators $(q_n)$ i.e. the subgroup $t_{(q_n)}(\T)$. From Theorem 4.3 \cite{KN} it follows that the sequence $(q_n\beta)$ is uniformly distributed modulo one for almost all $\beta \in \mathbb{R}$ in the sense of lebesgue measure. However note that if one considers the convergents $\frac{p_n}{q_n}$ of the continued fraction expansion of $\alpha$ and since $||q_n\alpha||_{\mathbb{Z}} \rightarrow 0$ (where $||.||_{\mathbb{Z}}$ is the distance from the integers), so $\alpha \notin \textit{W}_{\textbf{u}}$ (as mentioned above).

It has been observed already by Eggleston \cite{E} (followed by Larcher \cite{L}) and subsequently by Kraaikamp and Liardet \cite{KL} that
the asymptotic behavior of the sequence $(a_n)$, in particular the property of boundedness,
has a strong impact on the size of $t_{(q_n)}(\T)$ which we can express as the following two observations:
\begin{itemize}
\item[(O1)]  $t_{(q_n)}(\T)$ is at most enumerable if $(a_n)$  is bounded \cite{E}.
\item[(O2)] $|t_{(q_n)}(\T)|=\cont$ if $(a_n)$ is unbounded \cite{KL}.
\end{itemize}

  In a really impressive observation, Larcher \cite{L} proved in 1988 that "if the continued fraction expansion of an irrational $\alpha$ is bounded" (which means boundedness of the sequence $(a_n)$), i.e. in the case (O1) $$\{\beta \in \mathbb{R}: ||q_n\beta||_{\mathbb{Z}} \rightarrow 0\}$$
is nothing but the subgroup $\langle \alpha \rangle + \mathbb{Z}$ of $\mathbb{R}$. In other words, for any irrational $\alpha \in (0, 1)$, $t_{(q_n)}(\T) = \langle \alpha \rangle$ (instead of using the fractional part $\{x_j\}$ or working modulo 1, one can can think of the result in the circle group $\mathbb{R}\setminus \mathbb{Z} = \T$). The result of Larcher, namely Theorem 1, \cite{L} was generalized further in the sense that if $\eta_n = \eta_n(\alpha) = q_n|\theta_n|$ then for $\beta \in \mathbb{R}$, $\beta \in \langle \alpha \rangle + \mathbb{Z}$ if and only if $||q_n\beta|| \in o(\eta_n),~(n \rightarrow \infty)$ (Theorem $1'$ \cite{KL}). In \cite{KL} the authors also considered irrational numbers $\alpha$ with unbounded sequence $(a_n)$ and proved another very interesting result, namely Theorem 3, which forms the basis of the observation (O2). The investigation carried out in \cite{KL} was later strengthened in \cite{BDMW1} where, in particular, all the results were obtained in the circle.

The whole history of development of characterized subgroups can be found in the nice surveys \cite{D_NYC, DGDis} where almost all relevant references can be found. Very recently in \cite{DDB} the notion of characterized subgroups has been modified using the idea of natural density and in particular the consequences have been studied for arithmetic sequences.

For $m,n\in\mathbb{N}$ and $m\leq n$, let $[m, n]$ denotes the set $\{m, m+1, m+2,...,n\}$. By
 %For $n \in \omega$ we use the identification $n=\{0,1,\ldots,n-1\}$.
 $|A|$ we denote the cardinality of a set $A$. The lower and the upper natural densities of $A \subset \mathbb{N}$ are defined by
$$
\underline{d}(A)=\displaystyle{\liminf_{n\to\infty}}\frac{|A\cap [1,n]|}{n} ~~\mbox{and}~~
\overline{d}(A)=\displaystyle{\limsup_{n\to\infty}}\frac{|A\cap [1,n]|}{n}.
$$
If $\underline{d}(A)=\overline{d}(A)$, we say that the natural
density of $A$ exists and it is denoted by $d(A)$.
As usual,
$$
\mathcal{I}_d = \{A \subset \mathbb{N}: d(A) = 0\}
$$
 denotes the ideal of ``natural density zero" sets and $\mathcal{I}_d^*$ is the dual filter i.e. $\mathcal{I}_d^* = \{A \subset \mathbb{N}: d(A) = 1\}$.

 Let us now recall the notion of statistical convergence in the sense of \cite{F,Fr,S,St,Z} (see also \cite{B1,B2} for applications to Number Theory and Analysis).

\begin{definition}\label{Def1}
A sequence of real numbers $(x_n)$ is said to converge to a real number $x_0$ statistically if for any $\eps > 0$, $d(\{n \in \mathbb{N}: |x_n - x_0| \geq \eps\}) = 0$.
\end{definition}

It was proved in \cite{S} that $x_n \to x_0$ statistically  precisely when there exists a subset A of $ \N$ of asymptotic density 0, such that $\displaystyle{\lim_{n \in \N \setminus A}} x_n = x_0$.
Over the years, the notion of statistical convergence has been extended to general topological spaces using open neighborhoods \cite{MK} and in the last three decades a lot of work has been done on the notion of statistical convergence primarily because it extends the notion of usual convergence very naturally preserving many of the basic properties but at the same time including more sequences under its purview.

In order to relax the condition $a_nx \to 0 $ it thus seemed natural to involve the notion of statistical convergence. More precisely, $a_nx \to 0$ statistically means that for every $\eps > 0$ there exists a subset A of $ \N$ of asymptotic density 0, such that $\|a_nx\| < \eps$ for every $n \not \in A$, where  $\|x\|$ denotes the usual norm in $\T$ defined by the length of the shortest  arc connecting $x$ with 0.

 Using this notion, in \cite{DDB} the following definition has been introduced:

\begin{definition}
For a sequence of integers $(a_n)$ the subgroup
\begin{equation}\label{def:stat:conv}
t^s_{(a_n)}(\T) := \{x\in \T: a_nx \to 0\  \mbox{ statistically in }\  \T\}
\end{equation}
of $\T$ is called {\em a statistically characterized} (shortly, {\em an s-characterized}) $($by $(a_n))$ {\em subgroup} of $\T$.
\end{definition}

Clearly for a sequence of integers $(a_n)$, $t^s_{(a_n)}(\T) \supset t_{(a_n)}(\T)$. The following  result from \cite{DDB} justifies the investigation of this new notion of s-characterized subgroups as it has been established that, though in general, larger in size, these subgroups are still essentially topologically nice.\\

\noindent{\bf Theorem A.} \cite{DDB} {\em
%\begin{theorem}\label{theoremborel}
 $t^s_{(a_n)}(\T)$ is a $F_{\sigma\delta}$ (hence, Borel) subgroup of $\T$ containing $t_{(a_n)}(\T)$.}\\
%\end{theorem}

In this note, we continue to investigate this new notion but move in the direction of Eggleston, Larcher etc i.e. the statistically characterized subgroups generated by the sequence of denominators $(q_n)$ for the sequence of convergents $(\frac{p_n}{q_n})$ of an irrational number $\alpha$. Since we live in $\T$, we only consider irrational numbers from $[0, 1]$. One of our main observation is about the size of a s-characterized subgroup for any irrational number $\alpha$.\\

\noindent{\bf Theorem B.} (cf. Theorem B \cite{DDB})
%\label{theorem3}
Let $\alpha\in(0,1)$ be an irrational number. Then the associated s-characterized subgroup $t^s_{(q_n)}(\T)$ is uncountable.\\

 This theorem again points out the substantial amount of difference between characterized and s-characterized subgroups
by ``breaking" the dichotomy (O1)/(O2).
As a consequence  and some further observations, we obtain that the new subgroup $t^s_{(q_n)}(\T)$ always
differs from the subgroup $t_{(q_n)}(\T)$. In fact we can get the following assertion as a consequence.\\

\noindent{\bf Theorem C.} (cf. Theorem C \cite{DDB}) For any irrational number $\alpha \in (0, 1)$ we have $|t^s_{(q_n)}(\T)\setminus t_{(q_n)}(\T)| = \mathfrak{c}$.\\

We end this section by recalling the well-known inequalities and a canonical representation of real numbers in terms of $\theta_n$ as given below which will be used in our main results.

\begin{proposition}\label{1}\cite{BDS}
The following results  hold:\\
(i) $\frac{1}{(a_{n+1}+2)q_n^2}<|\alpha -\frac{p_n}{q_n}|<\frac{1}{a_{n+1}q_n^2}$;\\
(ii) $\exists \lambda\in(0,1)$ for which $|\theta_{k+j}|<\lambda ^j|\theta_k|$ and $q_{k-j}<\lambda ^jq_k$ for $j\geq 2$;\\
(iii) For any $k,m\in\N$ $\{q_k\theta_m\}=\{q_m\theta_k\}$ where $\{x\}$ denotes the fractional part of $x$.
\end{proposition}

\begin{proposition}\label{2}\cite{BDS}\cite{S}
Any $\beta\in(-\alpha ,1-\alpha )$ has a unique representation in the form
$$
\beta=\sum\limits_{k=0}^{\infty}d_k\theta_k
$$
where $0\leq d_0<a_1$, $0\leq d_k\leq a_{k+1}$ and $d_k=a_{k+1}\Rightarrow d_{k-1}=0$ for $k\geq 1$. Furthermore $d_{2i}\neq a_{2i+1}$ for infinitely many $i$.
\end{proposition}

 It is not hard to see that the above characterization is valid for elements of $\T$ also when the concerned $\alpha$ comes from [0,1] which we will repeatedly use in the next sections. Further we will denote by $supp(\beta)$ the set $\{k: d_k\neq 0\}$.\\
%Throughout the chapter, for any irrational number $\alpha$ we automatically associate with $\alpha$ the continued fraction approximation of it $[a_0;a_1,a_2,...]$,for any $n\in\N$, $\frac{p_n}{q_n}=[a_0;a_1,a_2,...,a_n]$ and $\theta_n=q_n\alpha -p_n$ unless stated otherwise.

\section{Proofs of Theorem B and C}

We start with the following important observation which will come to our help later to establish our main results.\\

\begin{lemma}\label{p2result1}
Let $\alpha\in(0,1)$ be an irrational number. Let
$$
\mathbb{I}:=\{\bigcup\limits_{i\in\N}[s_i,t_i]:i\in\N,s_i<t_i,s_{i+1}=t_i+2,\lim\limits_{i\to\infty}|t_i-s_i|=\infty\}.
$$
  Then there is a $\beta\in \T$ with the following property:

$(\star)$ For $\varepsilon>0$ there exists $B\subset\N$ with $d(B)=0$ such that whenever $n\notin B$ and $1\leq r\leq a_{n+1}$, we have $\|rq_n\beta\|<\varepsilon$.
\end{lemma}
\begin{proof}
Let us take
$$
\beta=\sum\limits_{k=1}^{\infty}d_k\theta_k,
$$
 where $(supp(\beta))^c=\{k:d_k=0\}=\bigcup\limits_{i=1}^{\infty}[s_i,t_i]$ and for any $k$, $d_k\neq 0\Rightarrow d_k=1$. We will show that $\beta$ has the desired property $(\star)$. For this, let $\lambda\in(0,1)$ be as in Proposition \ref{1}(ii). Let $\varepsilon>0$ be given. Choose $m\in\N$ such that $\frac{\lambda^m}{1-\lambda}<\frac{\varepsilon}{2}$. As $\lim\limits_{i\rightarrow\infty}|t_i-s_i|=\infty$, there is $i_0\in\N$ such that $|t_i-s_i|>2m$ for all $i\geq i_0$. Write $B_1=\{t_i+1:i\in\N\}$. As $\lim\limits_{i\to\infty}|t_{i+1}-t_i|=\infty$, so it readily follows that $d(B_1)=0$. Let $B_2=[1,s_{i_0}-1]$. Obviously $d(B_2)=0$. Now for each $i\in\N$, let $A_i=[s_{i+i_0},s_{i+i_0}+m-1]\cup [t_{i+i_0}-m+1,t_{i+i_0}]$ and write $B_3=\bigcup\limits_{i=1}^{\infty}A_i$. Observe that $d(B_3)=\lim\limits_{i\to\infty}\frac{2mi}{\sum\limits_{j=1}^{i}|t_j-s_j|}=0$. Let $B=B_1\cup B_2\cup B_3$. It is clear that $d(B)=0$.

  Let $n\in B^c$ and choose $r$, $1\leq r\leq a_{n+1}$. Then $rq_n\beta=rq_n\sum\limits_{i=1}^{\infty}\theta_{t_i+1}=
 rq_n\sum\limits_{\stackrel{i\in\N}{t_i+1<n}}\theta_{t_i+1}+rq_n\sum\limits_{\stackrel{i\in\N}{t_i+1\geq n}}\theta_{t_i+1}$, where the absolute value of the second summation
 $$
 |rq_n\sum\limits_{\stackrel{i\in\N}{t_i+1\geq n}}\theta_{t_i+1}|\leq rq_n\sum\limits_{\stackrel{i\in\N}{t_i+1\geq n}}|\theta_{t_i+1}|<rq_n\lambda^m|\theta_n|+rq_n\lambda^{3m}|\theta_n|+...
 $$
 $$
 =rq_n|\theta_n|(\lambda^m+\lambda^{3m}+\lambda^{5m}+...)=rq_n|\theta_n|\frac{\lambda^m}{1-\lambda^{2m}}<rq_n|\theta_n|\frac{\lambda^m}{1-\lambda}
 <\frac{r}{a_{n+1}}\frac{\varepsilon}{2}\leq \frac{\varepsilon}{2}.
 $$
 So $\|rq_n\sum\limits_{\stackrel{i\in\N}{t_i+1\geq n}}\theta_{t_i+1}\|<\frac{\varepsilon}{2}$.
Now if we write $\{t_i+1:i\in\N\}\cap [1,n-1]=\{j_1<j_2<...<j_N\} ~$ (say), then using Proposition \ref{1}(iii) we observe that the first summation
$$
rq_n\sum\limits_{\stackrel{i\in\N}{t_i+1<n}}\theta_{t_i+1}=rq_n\sum\limits_{i=1}^{N}\theta_{j_i}
=r[(l_1+q_{j_1}\theta_n)+(l_2+q_{j_2}\theta_n)+...+(l_N+q_{j_N}\theta_n)]=r.\sum\limits_{i=1}^{N}l_i+\sum\limits_{i=1}^{N}rq_{j_i}\theta_n
$$
for suitable choice of integers $l_i$, $i=1,2,...,N$ with
$$
|\sum\limits_{i=1}^{N}rq_{j_i}\theta_n|\leq \sum\limits_{i=1}^{N}|rq_{j_i}\theta_n|<\lambda^m|rq_n\theta_n|+\lambda^{m+2}|rq_n\theta_n|+...+\lambda^{m+2N}|rq_n\theta_n|
$$
$$
=|rq_n\theta_n|(\lambda^m+\lambda^{m+2}+...+\lambda^{m+2N})
<|rq_n\theta_n|(\lambda^m+\lambda^{m+2}+...\infty)\leq \frac{r}{a_{n+1}}\frac{\lambda^m}{1-\lambda^2}<\frac{\lambda^m}{1-\lambda}<\frac{\varepsilon}{2}.
$$
Hence $\|rq_n\sum\limits_{\stackrel{i\in\N}{t_i+1<n}}\theta_{t_i+1}\|<\frac{\varepsilon}{2}$, which implies that
$$
\|rq_n\beta\|=\|rq_n\sum\limits_{i=1}^{\infty}\theta_{t_i+1}\|\leq
 \|rq_n\sum\limits_{\stackrel{i\in\N}{t_i+1<n}}\theta_{t_i+1}\|+\|rq_n\sum\limits_{\stackrel{i\in\N}{t_i+1\geq n}}\theta_{t_i+1}\|<\varepsilon.
 $$
This proves the result. Further note that $\beta$ is unique for a given choice of a member in the class $\mathbb{I}$. This follows from the construction of $\beta$ where $supp(\beta)$ is determined by the member of $\mathbb{I}$ and in view of the unique representation of $\beta$ (see Proposition \ref{2}).
\end{proof}
\begin{lemma}\label{p2result2}
The class $\I$ defined in Lemma \ref{p2result1} is uncountable.
\end{lemma}
\begin{proof}
Note that for any $\nu>1$ we can choose $\bigcup\limits_{i\in\N}[s_i,t_i]\in\I$ such that $|t_i-s_i|=\lfloor i^\nu\rfloor$ and for different $\nu$ we get distinct members of $\mathbb{I}$.
\end{proof}

\begin{proposition}\label{suffcond1}
Let $\beta\in \T$ be such that $(supp(\beta))^c=\{k\in\N: d_k=0\}\in \mathbb{I}$ and for any $k$, $d_k\neq 0\Rightarrow d_k=1$. Then $\beta\in t^s_{(q_n)}(\T)$.
\end{proposition}
\begin{proof}
The conclusion readily follows from Lemma \ref{p2result1} taking $r=1$ for all $n$.
\end{proof}

\noindent\textbf{Proof of Theorem B}
\begin{proof}
Take $r=1$ for all $n$ in Lemma \ref{p2result1} and the result easily follows from Lemma \ref{p2result1} and Lemma \ref{p2result2}.
\end{proof}

\begin{remark}\label{r1}
 Let $\alpha=[0;a_1,a_2,...]$ be an irrational number with the sequence $(a_n)$ being bounded. It was established in  \cite{L} that $t_{(q_n)}(\T)=\langle\alpha\rangle$, the cyclic subgroup generated by $\alpha$. Clearly  $t^s_{(q_n)}(\T)$ is an uncountable subgroup containing $\langle\alpha\rangle$. Moreover, as $\langle\alpha\rangle$ is countable, we immediately get that $t^s_{(q_n)}(\T)\neq t_{(q_n)}(\T)$.
 \end{remark}

\begin{corollary}\label{p2result4}
Let $(f_n)$ be the celebrated Fibonacci sequence, i.e. $f_1=1, f_2=1$ and for all $n\geq 3$, $f_n = f_{n-1} + f_{n-2}$. Then the associated s-characterized subgroup $t^s_{(f_n)}(\T)$ is uncountable. This answers the open question \textbf{Question 6.4} from \cite{DDB}.
\end{corollary}
\begin{proof}
Let $\alpha=[0;1,1,1,...]$. Then the associated sequence $(q_n)$ of denominators of the convergents of the continued fraction approximation of $\alpha$ coincides with $(f_n)$. The result now follows from Theorem B.
\end{proof}

The situation when $\alpha=[0;a_1,a_2,...]$ is an irrational number with the sequence $(a_n)$ being unbounded seems to be much more complex as in that case the characterized subgroup $t_{(q_n)}(\T)$ itself happens to be uncountable \cite{KL}. However the following result shows that in that case also, the corresponding statistically characterized subgroup is strictly larger than the characterized subgroup.

\begin{theorem}\label{p2result5}
For any irrational number $\alpha=[0;a_1,a_2,...]$ with unbounded $(a_n)$, $t^s_{(q_n)}(\T) \varsupsetneq t_{(q_n)}(\T)$.
\end{theorem}
\begin{proof}
We will construct an element $\beta\in t^s_{(q_n)}(\T)$ which is outside $t_{(q_n)}(\T)$. Since $(a_n)$ is unbounded, we can choose a  subsequence $(a_{n_i})$ of $(a_n)$ such that $a_{n_i}>2$ for all $i$, $\lim\limits_{i\to\infty}a_{n_i}=\infty$ and also $\lim\limits_{i\to\infty}|n_{i+1}-n_i|=\infty$.
% Let $(a_{n_{i_k}})_k$ be any subsequence of $(a_{n_i})$. Then certainly we have $a_{n_{i_k}}>2$ for all $k$,$\lim\limits_{k\to\infty}a_{n_{i_k}}=\infty$ and $\lim\limits_{k\to\infty}|n_{i_{k+1}}-n_{i_k}|=\infty$. So for convenience in writing, let us unambiguously denote $(a_{n_{i_k}})$ by $(a_{n_i})$ throughout the result.

 Let us choose for any $i\in\N$, $s_i,t_i\in\N$ such that $s_1=1$, $t_1=n_2-2$; and in general $s_i=n_i$, $t_i=n_{i+1}-2$ for $i\geq 2$. Note that $\lim\limits_{i\to\infty}|t_i-s_i|=\lim\limits_{i\to\infty}|(n_{i+1}-n_i)-2|=\infty$. Now take $\beta=\sum\limits_{k=1}^{\infty}d_k\theta_k$ where  $(supp(\beta))^c=\{k:d_k=0\}=\bigcup\limits_{i=1}^{\infty}[s_i,t_i]$ and $d_{t_i+1}=\lfloor \frac{a_{t_i+2}}{2}\rfloor$ for all $i$. Writing in a clearer way, we have
 $$
 \beta=\sum\limits_{i=1}^{\infty}\lfloor\frac{a_{t_i+2}}{2}\rfloor\theta_{t_i+1}.
 $$

First we will show that $\beta\in t^s_{(q_n)}(\T)$. Take any $\varepsilon>0$, arbitrarily small. Then proceeding as in the proof of Lemma \ref{p2result1}, we can choose $\lambda\in(0,1)$, $m, i_0\in\N$, $B_k\subset\N$ for $k=1,2,3$ such that\\
(i) $\frac{\lambda^m}{1-\lambda}<\frac{\varepsilon}{2}$,\\
(ii) $|t_i-s_i|>2m$ for all $i\geq i_0$,\\
(iii) $d(B)=0$, where $B=B_1\cup B_2\cup B_3$ and
$$B_1=\{t_i+1:i\in\N\}, B_2=[1,s_{i_0}-1] \mbox{~~and~~} B_3=\bigcup\limits_{i=1}^{\infty}([s_{i+i_0},s_{i+i_0}+m-1]\cup [t_{i+i_0}-m+1,t_{i+i_0}]).
$$
Observe that for $n\in B^c$,
$$
q_n\beta=q_n\sum\limits_{i=1}^{\infty}\lfloor\frac{a_{t_i+2}}{2}\rfloor\theta_{t_i+1}
=q_n\sum\limits_{\stackrel{i\in\N}{t_i+1<n}}\lfloor\frac{a_{t_i+2}}{2}\rfloor\theta_{t_i+1}+q_n\sum\limits_{\stackrel{i\in\N}{t_i+1\geq n}}\lfloor\frac{a_{t_i+2}}{2}\rfloor\theta_{t_i+1}.
$$
Let $\{t_i+1:i\in\N\}\cap [n,\infty)=\{j_1<j_2<...<j_i<...\}$. Then the absolute value of the second summation
$$
|q_n\sum\limits_{\stackrel{i\in\N}{t_i+1\geq n}}\lfloor\frac{a_{t_i+2}}{2}\rfloor\theta_{t_i+1}|
=|q_n\sum\limits_{i=1}^{\infty}\lfloor\frac{a_{j_i+1}}{2}\rfloor\theta_{j_i}|
\leq q_n\sum\limits_{i=1}^{\infty}|\lfloor\frac{a_{j_i+1}}{2}\rfloor\theta_{j_i}|
\leq \frac{1}{2}\sum\limits_{i=1}^{\infty}a_{j_i+1}|q_n\theta_{j_i}|
$$
$$
<\frac{1}{2}\sum\limits_{i=1}^{\infty}a_{j_i+1}|\lambda^{(2i-1)m}q_{j_i}\theta_{j_i}|
<\frac{1}{2}\sum\limits_{i=1}^{\infty}\lambda^{(2i-1)m}a_{j_i+1}\frac{1}{a_{j_i+1}}
<\frac{1}{2}\frac{\lambda^m}{1-\lambda^{2m}}<\frac{1}{2}\frac{\lambda^m}{1-\lambda}<\frac{\varepsilon}{2}.
$$

Again writing $\{t_i+1:i\in\N\}\cap [1,n-1]=\{{j^\prime}_1<{j^\prime}_2<...<{j^\prime}_N\}$, and using Proposition \ref{1}(iii) we can observe that the first summation
$$
q_n\sum\limits_{\stackrel{i\in\N}{t_i+1< n}}\lfloor\frac{a_{t_i+2}}{2}\rfloor\theta_{t_i+1}
=\sum\limits_{i=1}^{N}\lfloor\frac{a_{j_i+1}}{2}\rfloor q_n\theta_{j_i}
=\sum\limits_{i=1}^{N}\lfloor\frac{a_{j_i+1}}{2}\rfloor (l_i+q_{j_i}\theta_n)
=\sum\limits_{i=1}^{N}\lfloor\frac{a_{j_i+1}}{2}\rfloor l_i+\sum\limits_{i=1}^{N}\lfloor\frac{a_{j_i+1}}{2}\rfloor q_{j_i}\theta_n 
$$
for suitable choice of integers $l_i$, $i=1,2,...,N$. Note that $\sum\limits_{i=1}^{N}\lfloor\frac{a_{j_i+1}}{2}\rfloor l_i$ is also an integer and
$$
|\sum\limits_{i=1}^{N}\lfloor\frac{a_{j_i+1}}{2}\rfloor q_{j_i}\theta_n|
\leq \frac{1}{2}\sum\limits_{i=1}^{N}a_{j_i+1}|q_{j_i}\theta_n|
< \frac{1}{2}\sum\limits_{i=1}^{N}a_{j_i+1}\lambda^{(2i-1)+m}|q_{j_i}\theta_{j_i}|
< \frac{1}{2}\sum\limits_{i=1}^{N}\lambda^{(2i-1)+m}a_{j_i+1}\frac{1}{{a_{j_i+1}}}
$$
$$
< \frac{1}{2}\sum\limits_{i=1}^{N}\lambda^{(2i-1)+m}<\frac{1}{2}\sum\limits_{i=1}^{\infty}\lambda^{(2i-1)+m}
<\frac{1}{2}\frac{\lambda^m}{1-\lambda^2}<\frac{1}{2}\frac{\lambda^m}{1-\lambda}<\frac{\varepsilon}{2}.
$$
So we can conclude that
$$
\|q_n\beta\|=\|q_n\sum\limits_{i=1}^{\infty}\lfloor\frac{a_{t_i+2}}{2}\rfloor\theta_{t_i+1}\|
\leq \|q_n\sum\limits_{\stackrel{i\in\N}{t_i+1<n}}\lfloor\frac{a_{t_i+2}}{2}\rfloor\theta_{t_i+1}\|+
\|q_n\sum\limits_{\stackrel{i\in\N}{t_i+1\geq n}}\lfloor\frac{a_{t_i+2}}{2}\rfloor\theta_{t_i+1}\|<\varepsilon.
$$

 Next we show that $\beta\notin t_{(q_n)}(\T)$. For this it is sufficient to show that for infinitely many $n$,  $\|q_n\beta\|\in[a,b]$ for some $a, b, ~\mbox{where}~0<a<b<1$. For $j\in\N$, we have
$$
q_{t_j+1}\beta=q_{t_j+1}\sum\limits_{i=1}^{\infty}\lfloor\frac{a_{t_i+2}}{2}\rfloor\theta_{t_i+1}
$$
$$
=q_{t_j+1}\lfloor\frac{a_{t_j+2}}{2}\rfloor\theta_{t_j+1}+q_{t_j+1}\sum\limits_{i<j}\lfloor\frac{a_{t_i+2}}{2}\rfloor\theta_{t_i+1}
+q_{t_j+1}\sum\limits_{i>j}\lfloor\frac{a_{t_i+2}}{2}\rfloor\theta_{t_i+1}.
$$
As $\lim\limits_{i\to\infty}|t_i-s_i|=\infty$, so we can find $~i_0\in\N$ such that for all $i,j\geq i_0$, $i\neq j$ we have $|t_j-t_i|>2m$. Consequently using similar technique used in the first part of the proof, we can conclude  that for $j>i_0$
$$
|q_{t_j+1}\sum\limits_{i<j}\lfloor\frac{a_{t_i+2}}{2}\rfloor\theta_{t_i+1}
+q_{t_j+1}\sum\limits_{i>j}\lfloor\frac{a_{t_i+2}}{2}\rfloor\theta_{t_i+1}|<\varepsilon,
$$
and
$$
|q_{t_j+1}\lfloor\frac{a_{t_j+2}}{2}\rfloor\theta_{t_j+1}|
\leq \frac{1}{2}a_{t_j+2}|q_{t_j+1}\theta_{t_j+1}|<\frac{1}{2}a_{t_j+2}\frac{1}{a_{t_j+2}}=\frac{1}{2}.
$$
This shows that for all $j>i_0$ we have $|q_{t_j+1}\beta|<\frac{1}{2}+\varepsilon$. Since this is true for any $\varepsilon >0$, hence  $|q_{t_j+1}\beta|\leq \frac{1}{2}$. Further note that
$$
|q_{t_j+1}\beta|=|q_{t_j+1}\sum\limits_{i=1}^{\infty}\lfloor\frac{a_{t_i+2}}{2}\rfloor\theta_{t_i+1}|
=|q_{t_j+1}\lfloor\frac{a_{t_j+2}}{2}\rfloor\theta_{t_j+1}+q_{t_j+1}\sum\limits_{i\neq j}\lfloor\frac{a_{t_i+2}}{2}\rfloor\theta_{t_i+1}|
$$
$$
\geq |q_{t_j+1}\lfloor\frac{a_{t_j+2}}{2}\rfloor\theta_{t_j+1}|-|q_{t_j+1}\sum\limits_{i\neq j}\lfloor\frac{a_{t_i+2}}{2}\rfloor\theta_{t_i+1}|>|q_{t_j+1}\frac{a_{t_j+2}-1}{2}\theta_{t_j+1}|-\varepsilon
$$
$$
=\frac{a_{t_j+2}-1}{2}|q_{t_j+1}\theta_{t_j+1}|-\varepsilon
>\frac{a_{t_j+2}-1}{2}\frac{1}{a_{t_j+2}+2}-\varepsilon=\frac{1}{2}(1-\frac{3}{a_{t_j+2}+2})-\varepsilon
$$
$$
\geq \frac{1}{2}(1-\frac{3}{4})-\varepsilon =\frac{1}{8}-\varepsilon .
$$
Since this is true for any $\varepsilon >0$, so $|q_{t_j+1}\beta|\geq \frac{1}{8}$.
Hence as desired, we get that for all $j>i_0$, $\|q_{t_j+1}\beta\|\in [\frac{1}{8} ,\frac{1}{2}]$, i.e. $\|q_n\beta\|$ can not converge to $0$.
\end{proof}

\noindent\textbf{Proof of Theorem C}
\begin{proof} The result is obvious for the irrational numbers $\alpha=[0;a_1,a_2,...]$ with bounded $(a_n)$ in view of Remark \ref{r1}. Even when $(a_n)$ is unbounded this can be seen from Theorem \ref{p2result5}. Observe that for every possible subsequence of the subsequence $(a_{n_i})$ of $(a_n)$ used in the proof (which automatically has the desired properties), we can proceed similarly to construct an element $\beta\in t^s_{(q_n)}(\T)\setminus t_{(q_n)}(\T)$. Note that for two distinct subsequences of $(a_{n_i})$, the corresponding constructed elements $\beta$'s of $\T$ are distinct as each subsequence gives a unique element in $\mathbb{I}$ which in turn determines the support of the corresponding element $\beta$. As the cardinality of the set of all possible distinct subsequences of $(a_{n_i})$ is $\mathfrak{c}$, it follows that in this case also $|t^s_{(q_n)}(\T)\setminus t_{(q_n)}(\T)| = \mathfrak{c}$.
\end{proof}
%for different subsequence $(a_{n_{i_k}})_k$ of $(a_{n_i})$ we get a different $\{(s_{i_k},t_{i_k}):k\in\N\}$ with $\bigcup\limits_{k\in\N}[s_{i_k},t_{i_k}]\in\I$ and can construct a different $\beta\in t^s_{(q_n)}(\T)\setminus t_{(q_n)}(\T)$ whereas the number of different subsequences of $(a_{n_i})$ is actually $\mathfrak{c}$.
Now a natural question arises, whether it is possible to point out certain characteristics of the elements of $\T$ which ensures its inclusion in the statistically charaterized subgroup corresponding to a given irrational number and our last result of this section provides a positive answer in that direction.\\

\begin{proposition}\label{suff}
Let $\alpha$ be an irrational number and $\beta\in\T$ be such that $d(supp(\beta))=0$. Then $\beta\in t^s_{(q_n)}(\T)$.
\end{proposition}
\begin{proof}
Let
$$
\beta=\sum\limits_{k=1}^{\infty}d_k\theta_k,
$$
where $d(supp(\beta))=d(k: d_k\neq 0\})=0$. Let $\lambda\in (0,1)$ be as in Proposition \ref{2}. For a given  $\varepsilon > 0$ first choose $m\in\N$ such that $\frac{\lambda^m}{1-\lambda}<\frac{\varepsilon}{2}$. Let $A^*=\bigcup\limits_{i=-m}^{m}(supp(\beta)+i)$. Note that $d(A^*)=0$. Further note that  $n\notin A^*\Rightarrow n+i\notin supp(\beta)$ for all $-m\leq i \leq m$. So whenever $n\notin A^*$, we have
$$
q_n\beta=q_n\sum\limits_{k\in supp(\beta)}d_k\theta_k=q_n\sum\limits_{\stackrel{k\in supp(\beta)}{k<n}}d_k\theta_k
+q_n\sum\limits_{\stackrel{k\in supp(\beta)}{k>n}}d_k\theta_k.
$$
Let $supp(\beta)\cap (n,\infty)=\{j_1<j_2<...<j_i<...\}$. Now using Proposition \ref{1}, the absolute value of the second summation is
$$
|q_n\sum\limits_{\stackrel{k\in supp(\beta)}{k>n}}d_k\theta_k|
\leq q_n\sum\limits_{i=1}^{\infty}d_{j_i}|\theta_{j_i}|
\leq \sum\limits_{i=1}^{\infty}d_{j_i}|q_n\theta_{j_i}|
$$
$$
< \sum\limits_{i=1}^{\infty}d_{j_i}|\lambda^{m+i}q_{j_i}\theta_{j_i}|
<\sum\limits_{i=1}^{\infty}d_{j_i}\lambda^{m+i}\frac{1}{a_{j_i+1}}
$$
$$
=\sum\limits_{i=1}^{\infty}\frac{d_{j_i}}{{a_{j_i+1}}}\lambda^{m+i}\leq \frac{\lambda^m}{1-\lambda}<\frac{\varepsilon}{2}.
$$
On the other hand, if $supp(\beta)\cap [1,n)=\{j^\prime_1<j^\prime_2<...<j^\prime_N\}$, then the first summation  becomes
$$
q_n\sum\limits_{\stackrel{k\in supp(\beta)}{k<n}}d_k\theta_k
= q_n\sum\limits_{i=1}^{N}d_{j_i}\theta_{j_i}
=\sum\limits_{i=1}^{N}d_{j_i}q_n\theta_{j_i}
=\sum\limits_{i=1}^{N}d_{j_i}(l_i+q_{j_i}\theta_n)=\sum\limits_{i=1}^{N}d_{j_i}l_i+\sum\limits_{i=1}^{N}d_{j_i}q_{j_i}\theta_n
$$
for some integers $l_i, i=1,2,...,N$. Note that $\sum\limits_{i=1}^{N}d_{j_i}l_i$ is an integer and
$$
|\sum\limits_{i=1}^{N}d_{j_i}q_{j_i}\theta_n|
\leq \sum\limits_{i=1}^{N}d_{j_i}|q_{j_i}\theta_n|
\leq \sum\limits_{i=1}^{N}d_{j_i}\lambda^{i+m}|q_{j_i}\theta_{j_i}|
$$
$$
< \sum\limits_{i=1}^{N}\frac{d_{j_i}}{a_{j_i+1}}\lambda^{i+m}\leq \frac{\lambda^m}{1-\lambda}<\frac{\varepsilon}{2}.
$$
It is now evident that
$$
\|q_n\beta\|\leq \|q_n\sum\limits_{\stackrel{k\in supp(\beta)}{k<n}}d_k\theta_k\|
+|q_n\sum\limits_{\stackrel{k\in supp(\beta)}{k>n}}d_k\theta_k\|<\varepsilon.
$$
\end{proof}

\section{Some observations regarding characterizing sequences}

A sequence $(r_n)$ of positive integers is called "a characterizing sequence" for a real number $\alpha$ \cite{BDS} if for any real number $\beta$, one has
\begin{center}
$ (3.1) ~~~~~~\lim_{n \rightarrow \infty} ||r_n\beta|| = 0~\mbox{if and only if}~\beta \in \langle \alpha\rangle + \mathbb{Z}.$
\end{center}
If $\alpha$ is a rational number of the form $\frac{p}{q}$ where $p$ and $q$ are coprimes, then simply the sequence $(qn)_{n \in \mathbb{Z}}$ is the required characterizing sequence (for details, see the discussions in Section 4 of \cite{BDS}). So the actual interesting question arises when $\alpha$ is an irrational number. The result of Larcher \cite{L} is so significant because for the first time it was established that for an irrational number $\alpha$ with bounded sequence $(a_n)$, where $\alpha = [a_0; a_1, a_2, a_3, \dots]$, the sequence of denominators $(q_n)$ in the sequence of convergents $(\frac{p_n}{q_n})$ happens to be a characterizing sequence of $\langle \alpha\rangle + \mathbb{Z}$. However the problem remained open for irrational numbers $\alpha$ with unbounded sequence $(a_n)$ for long time until B\' \i r\' o, Deshouillers and S\' os \cite{BDS} showed that it is indeed possible to construct a characterizing sequence for such irrational numbers also. They actually established the important fact
that every countable subgroup of $\T$ can be  characterized. Our primary interest in this section is the sequence constructed in the following remarkable result.

\begin{theorem}\cite{BDS}\label{3}
Let $\alpha$ be an irrational number and $(q_n)$ be the associated sequence of denominators of the convergents of continued fraction approximation of $\alpha$. Let $(x_n)$ be the increasing sequence formed by the elements of the set
$$
\{rq_n:1\leq r\leq a_{n+1},n\in\N\}.
$$
Then the characterized subgroup $t_{(x_n)}(\T)$ is countably infinite. More precisely, $t_{(x_n)}(\T)=\langle\alpha\rangle$.
\end{theorem}

 Summing up all the relevant results of \cite{L} and \cite{BDS}, we can say that for any irrational number $\alpha \in (0,1)$ with bounded $(a_n)$, we have
 \begin{center}
 $ (3.2)~~~~~~t_{(q_n)}(\T) = t_{(x_n)}(\T)$
 \end{center}
 while for any irrational number $\alpha \in (0,1)$ with unbounded $(a_n)$, we have
 \begin{center}
 $ (3.3)~~~~~~t_{(q_n)}(\T) \varsupsetneq t_{(x_n)}(\T).$
 \end{center}

 When the new notion of statistically characterized subgroups come into picture, the question as to how to approach the subgroup $\langle \alpha\rangle + \mathbb{Z}$ of $\mathbb{R}$ (or the subgroup $\langle \alpha\rangle$ of $\T$, where $\alpha \in (0, 1)$) through some statistically characterizing sequence in the same sense as in (3.1) is trivial, because usual convergence of a sequence always implies statistical convergence of that sequence with same limit.

 What we intend to do in this section is to carry out a similar comparative study of two naturally arising s-characterized subgroups for these characterizing sequences constructed above in \cite{BDS}, one for $(q_n)$ and the other for $(x_n)$ to be precise.

\begin{lemma}\label{new4}
(folklore) Let $(a_n)$ be a sequence of positive real numbers, $(b_n)$ be a sequence of positive numbers and there exists $c>0$ such that $\liminf\limits_{n\to\infty}(a_nb_n)\geq c$. Then $\lim\limits_{n\to\infty}a_n=0\Rightarrow \lim\limits_{n\to\infty}b_n=\infty$.
\end{lemma}
\begin{proof}
Let us take $L>0$ arbitrarily large. Choose $n_1\in\N$ such that $a_n<\frac{c}{L+1}$ for all $n\geq n_1$. Again, $\exists ~n_2\in\N$ such that $a_nb_n>c-\frac{c}{L+1}$ for all $n\geq n_2$. Let $n_0=\max\{n_1,n_2\}$. Then for $n\geq n_0$ we have
$$
b_n>\frac{1}{a_n}(c-\frac{c}{L+1})>\frac{L+1}{c}(c-\frac{c}{L+1})=(L+1)-1=L,
$$
i.e. $\lim\limits_{n\to\infty}b_n=\infty$.
\end{proof}

\begin{proposition}\label{new1}
Let $\alpha =[0;a_1,a_2,...]$ be an irrational number with bounded $(a_n)$ and $(x_n)$ be the sequence defined as above i.e. by $\{rq_n:1\leq r\leq a_{n+1},n\in\N\}$. Then $t^s_{(q_n)}(\T)= t^s_{(x_n)}(\T)$.
\end{proposition}
\begin{proof}
Let us take any $\beta\in t^s_{(q_n)}(\T)$. Let  $|a_n|\leq M~\forall n$ for some $M>0$. For $\varepsilon >0$ there is $B\subset\N$ with $d(B)=0$ such that whenever $n\notin B$, we have $\|q_n\beta\|<\frac{\varepsilon}{M}$.
Construct $A\subset\N$ as follows.
$$
 A:=\{m:x_m=rq_n;n\in B,1\leq r\leq a_{n+1}\}.
 $$
Let $B=\{n_1<n_2<...<n_i<...\}$. Clearly $d(A)$ attains its maximum value when $a_{n_i}=M$ for all $i$ and $a_n=1$ elsewhere. We show that $d(A)=0$ even in that case. Note that $d(B)=0\Rightarrow\lim\limits_{i\to\infty}\frac{i}{n_i}=0$. Then
$$
d(A)\leq \lim\limits_{i\to\infty}\frac{Mi}{(n_1-1)+M+(n_2-n_1-1)+M+...(n_i-n_{i-1})+M}
$$
$$
=\lim\limits_{i\to\infty}\frac{Mi}{Mi+(n_i-i)}=\lim\limits_{i\to\infty}\frac{1}{1+\frac{n_i-i}{Mi}}
=\lim\limits_{i\to\infty}\frac{1}{1+\frac{1}{M}(\frac{n_i}{i}-1)}=0.
$$
Now observe that s $m\notin A\Rightarrow x_m=rq_n$ for some $n\notin B$ and $1\leq r\leq a_{n+1}$, and consequently we obtain
$$
\|x_m\beta\|=\|rq_n\beta\|\leq r\|q_n\beta\|<M\frac{\varepsilon}{M}=\varepsilon,
$$
 whenever $m\notin A$. This implies that $\beta\in t^s_{(x_n)}(\T)$.

  On the other hand, let $\beta\in t^s_{(x_n)}(\T)$. Let $\varepsilon>0$ be given. Then $\exists ~A\subset\N$ with $d(A)=0$ with the property that whenever $n\notin A$, we have $\|x_n\beta\|<\varepsilon$. Let $B=\{n: x_n=q_m \mbox{~for some~} m\}$. As the sequence $(a_n)$ is bounded, let $|a_n|\leq M~\forall n$ for some $M>0$. A little computation gives us that $\underline{d}(B)\geq \frac{1}{M}$. Let $B=\{n_i:i\in\N\}$ and $B\cap A =\{n_{\phi_i}:i\in\N\}$, for some increasing sequences $(n_i)$ and $(\phi_i)$. Clearly we have $d(B\cap A)=\lim\limits_{i\to\infty}\frac{i}{n_{\phi_i}}=0$. Now we have $\liminf\limits_{i\to\infty}\frac{i}{n_i}\geq \frac{1}{M}$. So $\liminf\limits_{i\to\infty}\frac{\phi_i}{n_{\phi_i}}\geq \liminf\limits_{i\to\infty}\frac{i}{n_i}\geq\frac{1}{M}$. But $\frac{\phi_i}{n_{\phi_i}}=\frac{i}{n_{\phi_i}}.\frac{\phi_i}{i}$ for all $i$, and $\lim\limits_{i\to\infty}\frac{i}{n_{\phi_i}}=0$. Hence, by using Lemma \ref{new4} we conclude that $\lim\limits_{i\to\infty}\frac{\phi_i}{i}=\infty$, i.e. $\lim\limits_{i\to\infty}\frac{i}{\phi_i}=0$.

 Let $C=\{\phi_i:i\in\N\}$. Then we have $d(C)=0$ and it is now evident that whenever $n\notin C$, $\|q_n\beta\|<\varepsilon$. Hence $\beta\in t^s_{(q_n)}(\T)$. This completes the proof of our assertion.

\end{proof}

Finally we deal with the natural question that whether Proposition \ref{new1} remains valid for irrational numbers $\alpha$ for which the sequence $(a_n)$ happens to be unbounded and the next proposition answers it in negative.

\begin{proposition}\label{p2result9}
There exists $\alpha=[0;a_1,a_2,...]$, an irrational number with the sequence $(a_n)$ being unbounded for which $t^s_{(q_n)}(\T)\neq t^s_{(x_n)}(\T)$.
\end{proposition}
\begin{proof}
Let $a_{n^2}=4n^2$ and $a_n=1$ elsewhere. Clearly $|(n+1)^2-n^2|=2n+1\to\infty$ as $n\to\infty$. Let us take any subsequence $(n_i^2)$ of $(n^2)$ such that $\lim\limits_{i\to\infty}\frac{i}{n_i}\in (0,\frac{1}{2})$. Such subsequence surely exists, e.g. take $n_i=\lfloor 3i\rfloor$ for $i\in\N$. Further, without any loss of generality we can take $n_1>1$. Now let us define for $i\in\N$, $s_i,t_i\in\N$ such that $s_1=1, t_1=n_1^2-2; s_i=n_i^2,t_i=n_{i+1}^2-2$ for all $i\geq 2$. Now we construct $\beta\in \T$ as follows.
$$
\beta=\sum\limits_{k=1}^{\infty}d_k\theta_k,
$$
where $(supp(\beta))^c=\{k:d_k=0\}=\bigcup\limits_{i=1}^{\infty}[s_i,t_i]$ and for $i\in\N$ $d_{t_i+1}=1$.  We will show that $\beta\in t^s_{(q_n)}(\T)\setminus t^s_{(x_n)}(\T)$.

 As $\lim\limits_{i\to\infty}|n_{i+1}^2-n_i^2|\geq \lim\limits_{n\to\infty}|(n+1)^2-n^2|=\infty$ so $\bigcup\limits_{i\in\N}[s_i,t_i]\in\I$ as in Lemma \ref{p2result1} and thus $\beta\in t^s_{(q_n)}(\T)$.
Now we show that $\beta\notin t^s_{(x_n)}(\T)$. Let us construct $B\subset\N$ such that
$$
B:=\{n\in\N:x_n=rq_{t_i+1};i\in\N,\frac{1}{4}a_{t_i+2}\leq r\leq \frac{3}{4}a_{t_i+2}\}.
$$
Our first claim is that $\overline{d}(B)>0$. To see this, observe that
$$
\overline{d}(B)=\lim\limits_{i\to\infty}\frac{2(n_1^2+n_2^2+...+n_i^2)}{4\sum\limits_{k=1}^{i}n_k^2+\sum\limits_{k=2}^{n_1-1}(4k^2+2k)
+\sum\limits_{k=n_1+1}^{n_2-1}(4k^2+2k)+...+\sum\limits_{k=n_{i-1}+1}^{n_i-1}(4k^2+2k)}
$$
$$
=\lim\limits_{i\to\infty}\frac{2(n_1^2+n_2^2+...+n_i^2)}{4\sum\limits_{k=2}^{n_i}k_2+2\sum\limits_{k=2}^{n_i-1}k}
=\lim\limits_{i\to\infty}\frac{2(n_1^2+n_2^2+...+n_i^2)}{n_i(n_i-1)-2+\frac{4n_i(n_i+1)(2n_i+1)}{6}-4}
$$
$$
=\lim\limits_{i\to\infty}\frac{2(\frac{n_1^2}{n_i^2}+\frac{n_2^2}{n_i^2}+...+\frac{n_i^2}{n_i^2})}{(1-\frac{1}{n_i}-\frac{6}{n_i^2})+\frac{2}{3}(2n_i+1)(1+\frac{1}{n_i})}
\geq \lim\limits_{i\to\infty}\frac{2(\frac{n_i^2}{n_i^2}+\frac{n_i^2}{n_i^2}+...+\frac{n_i^2}{n_i^2})}{1+\frac{2}{3}(2n_i+1)}
$$
$$
=\lim\limits_{i\to\infty}\frac{2(1+1+...+1)(i\mbox{~~times~~})}{1+\frac{2}{3}(2n_i+1)}
=\lim\limits_{i\to\infty}\frac{2i}{1+\frac{2}{3}(2n_i+1)}=\lim\limits_{i\to\infty}\frac{1}{\frac{5}{6i}+\frac{2}{3}\frac{n_i}{i}}
=\lim\limits_{i\to\infty}\frac{3}{2}\frac{i}{n_i}.
$$
Now as $\lim\limits_{i\to\infty}\frac{i}{n_i}\in (0,\frac{1}{2})$ so we get $\overline{d}(B)>0$.

 Finally, we show that there is some $B^\prime\subset\N$ with $\overline{d}(B^\prime)>0$ such that for $n\in B^\prime$, $|x_n\beta|\in[a,b]$ for some $0<a<b<1$  from which it will readily follow that $\|x_n\beta\|$ can not converge statistically to $0$. Let us take $\varepsilon>0$ arbitrarily small. Then for $n\in B$ and $j\in\N$
 $$
 x_n\beta=rq_{t_j+1}\sum\limits_{i=1}^{\infty}\theta_{t_i+1}=rq_{t_j+1}\theta_{t_j+1}
 =rq_{t_j+1}\sum\limits_{\stackrel{i\in\N}{i\neq j}}\theta_{t_i+1}.
 $$
 Now following the similar technique used in the proof of Lemma \ref{p2result1} and the last portion of Theorem \ref{p2result5}, it is easy to see that $\exists m,i_0\in\N$ for which
 $$
 j>i_0,i\in \N \mbox{~~and~~} i\neq j\Rightarrow |t_j-t_i|>2m,
 $$
  from which we eventually get that for all $j>i_0$,
 $$
 |rq_{t_j+1}\sum\limits_{\stackrel{i\in\N}{i\neq j}}\theta_{t_i+1}|\leq |rq_{t_j+1}\sum\limits_{\stackrel{i\in\N}{i< j}}\theta_{t_i+1}|+|rq_{t_j+1}\sum\limits_{\stackrel{i\in\N}{i>j}}\theta_{t_i+1}|<\varepsilon.
 $$
  Let
 $$
 B^\prime =\{n\in\N:x_n=rq_{t_i+1};i>i_0,\frac{1}{4}a_{t_i+2}\leq r\leq \frac{3}{4}a_{t_i+2}\}.
 $$
 As $(a_{n^2})$ is strictly increasing and tends to $\infty$, we have $d(B\setminus B^\prime)=0$ i.e. $\overline{d}(B^\prime)>0$.
 Now for $n\in B^\prime$
 $$
 |x_n\beta|\leq |rq_{t_j+1}\theta_{t_j+1}|
 +|rq_{t_j+1}\sum\limits_{\stackrel{i\in\N}{i\neq j}}\theta_{t_i+1}|
 < \frac{3}{4}a_{t_j+2}|q_{t_j+1}\theta_{t_j+1}|+\varepsilon <\frac{3}{4}a_{t_j+2}\frac{1}{a_{t_j+2}}+\varepsilon=\frac{3}{4}+\varepsilon.
 $$
On the other hand
$$
|x_n\beta|=|rq_{t_j+1}\theta_{t_j+1}
 +rq_{t_j+1}\sum\limits_{\stackrel{i\in\N}{i\neq j}}\theta_{t_i+1}|
\geq |rq_{t_j+1}\theta_{t_j+1}|
 -|rq_{t_j+1}\sum\limits_{\stackrel{i\in\N}{i\neq j}}\theta_{t_i+1}|
$$
$$
>|rq_{t_j+1}\theta_{t_j+1}|-\varepsilon \geq \frac{1}{4}a_{t_j+2}|q_{t_j+1}\theta_{t_j+1}|-\varepsilon
\geq \frac{1}{4}a_{t_j+2}\frac{1}{a_{t_j+2}+2}-\varepsilon
$$
$$
=\frac{1}{4}(1-\frac{1}{a_{t_j+2}+2})-\varepsilon
=\frac{1}{4}(1-\frac{1}{a_{n_{j+1}^2}+2})-\varepsilon\geq \frac{1}{4}(1-\frac{1}{a_4+2})-\varepsilon
$$
$$
=\frac{1}{4}(1-\frac{1}{18})-\varepsilon=\frac{17}{72}-\varepsilon.
$$
So for $n\in B^\prime$, $|x_n\beta|\in [\frac{17}{72}-\varepsilon ,\frac{3}{4}+\varepsilon ]$. Since $\varepsilon >0$ is arbitrary, hence it follows that $|x_n\beta|\in [\frac{17}{72},\frac{3}{4}]$ for all $n\in B^\prime$ which implies $\|x_n\beta\|$ can not converge statistically to $0$.

\end{proof}

We can actually say more about the difference $t^s_{(q_n)}(\T)\setminus t^s_{(x_n)}(\T)$.

\begin{corollary}\label{new5}
For the irrational number $\alpha$ used in Proposition \ref{p2result9}, $|t^s_{(q_n)}(\T)\setminus t^s_{(x_n)}(\T)|=\mathfrak{c}$.
\end{corollary}
\begin{proof}
In the proof of Proposition \ref{p2result9}, note that there are uncountably many of those subsequences $(n_i)$ which can be used to obtain the desired element $\beta$ and for different subsequences $(n_i)$, the resulting elements $\beta$'s are always distinct, e.g. for any $\nu >2$ take $n_i=\lfloor\nu i\rfloor$ for $i\in\N$.
\end{proof}

%Corollary \ref{new5} leads to following natural questions.
%\begin{problem}\label{prob1}
%For $\alpha$ used in Proposition \ref{p2result9}, find the exact size of the subgroup $t^s_{(x_n)}(\T)$.
%\end{problem}
%More generally,
\noindent\textbf{Concluding Remarks.} In \cite{DDB}, while starting out the investigation of the new notion of statistically characterized subgroups, several open problems were suggested, mostly regarding arithmetic sequences. As we conclude our investigations of this article, apart from getting new results, we have also been able to present answer to one of the open questions posed there, namely \textbf{Question 6.4}. Further in view of Proposition \ref{p2result9} the following two questions naturally arises.
%However it seems clear that the investigation of s-characterized subgroups do not follow, in general, the same path of characterized groups and the presence of natural density in the construction of this new notion makes it much more tougher to predict or handle. Under the circumstances, the following questions seems very natural which should be dealt with in order to throw more light into the investigation of s-characterized groups.

\begin{problem}\label{prob4}
Is the subgroup  $t^s_{(x_n)}(\T)$ in Proposition \ref{p2result9} countably infinite? If not, then does there exist any increasing sequence of positive integers $(y_n)$ such that $t^s_{(y_n)}(\T)$ is countably infinite?
\end{problem}

\begin{problem}\label{prob5}
Is the result (3.3) valid for s-characterized subgroups of $\T$ also i.e. for any irrational number $\alpha \in (0,1)$ with unbounded $(a_n)$, we have $t^s_{(q_n)}(\T) \varsupsetneq t^s_{(x_n)}(\T)$.
\end{problem}

If the answer to Problem \ref{prob5} is in affirmative then the following problem seems natural.

 \begin{problem}\label{prob6}
       Is it true that for any irrational number $\alpha \in (0,1)$ with unbounded $(a_n)$, $|t^s_{(q_n)}(\T) \setminus t^s_{(x_n)}(\T)| = \mathfrak{c}$.
\end{problem}

%On the other hand if the answer to Problem 3.2 is negative then one should consider the following question.

%\begin{problem}\label{prob2}
%Is there an irrational number $\alpha=[0;a_1,a_2,...]$ with unbounded $(a_n)$, for which $t^s_{(x_n)}(\T) = t^s_{(q_n)}(\T)$.
%\end{problem}

%\begin{problem}\label{prob3}
% Is it possible to characterize the subgroups $t^s_{(q_n)}(\T)$ and $t^s_{(x_n)}(\T)$ in term of support and the quotient $d_k$, $k\in\N$ used in the unique canonical representation of $\beta\in \T$.
%\end{problem}

\end{document}